\documentclass[11pt]{article}       
\usepackage{array}
\usepackage{amsmath}
\usepackage{amssymb}
\begin{document}

\newcommand{\comment}[1]{}    
\newcommand{\hs}{\enspace}
\newcommand{\hhs}{\thinspace}
\newcommand{\real}{\ifmmode {\rm R} \else ${\rm R}$ \fi}
\def\nat{\mbox{\vrule height 7pt width .7pt depth 0pt\hskip -.5pt\bf N}}
\newcommand{\qed}{\hfill{\setlength{\fboxsep}{0pt}
                  \framebox[7pt]{\rule{0pt}{7pt}}} \newline}

\newcommand{\eqed}{\qquad{\setlength{\fboxsep}{0pt}
                  \framebox[7pt]{\rule{0pt}{7pt}}}\newline }
\newtheorem{theorem}{Theorem}
\newtheorem{lemma}[theorem]{Lemma}         
\newtheorem{corollary}[theorem]{Corollary}
\newtheorem{definition}[theorem]{Definition}
\newtheorem{claim}[theorem]{Claim}%
\newtheorem{conjecture}[theorem]{Conjecture}
\newtheorem{proposition}[theorem]{Proposition}
\newtheorem{construction}[theorem]{Construction}
\newtheorem{problem}[theorem]{Problem}

\newcommand{\proof }{{\bf Proof: }}          



\newcommand{\seq}{d_1, \ldots, d_{2k+2}}
\def\eps{\varepsilon}
\newcommand{\eee}{{\mathbb E}}
\newcommand{\seqn}{d_1, \ldots, d_n}
\newcommand{\seqg}{d_1\geq \cdots  \geq d_n}
\newcommand{\edge}{\leftrightarrow}
\newcommand{\nedge}{\nleftrightarrow}
\newcommand{\pakg}{potentially $A_k$-graphical}
\newcommand{\ppkg}{potentially $P_k$-graphical}
\newcommand{\leqk}{\leq _k}
\newcommand\gbin[2]{\genfrac{[}{]}{0pt}{}{#1}{#2}}
\def\n{\notag}
\def\noedge{\not\leftrightarrow}
\def\f{f_d^k(G)}
\def\fn{f_d^k(K_n)}
\def\fab{f_d^k(K_{a,b})}
\def\to{\rightarrow}
\def\lf{\lfloor}
\def\rf{\rfloor}
\def\lc{\lceil}
\def\rc{\rceil}
\def\Dt{\Delta}
\def\ds{\displaystyle}
\def\tr{\tilde r}
\def\field{\mathbb{F}_q}
\def\F{\mathbb{F}}
\def\O{\mathcal{O}}
\def\P{\mathcal{P}}
\def\PP{\mathcal{P}}
\def\GG{\mathcal{G}}
\def\AA{\mathcal{A}}
\def\BB{\mathcal{B}}
\def\HH{\mathcal{H}}
\def\CC{\mathcal{C}}
\def\Q{\mathcal{Q}}
\def\QQ{\mathcal{Q}}
\def\e{\varepsilon}
\def\cp{\mbox{cp}}

\title{The de Bruijn-Erd\H os Theorem for \\ Hypergraphs}
\author{
Noga Alon
\thanks{Schools of Mathematics and Computer Science,
Sackler Faculty of Exact Sciences, Tel Aviv University,
Tel Aviv 69978, Israel. Email: {\tt nogaa@tau.ac.il}
Research supported in part by an ERC Advanced
grant and
by a USA-Israeli BSF grant} \quad
Keith E. Mellinger \thanks{Department of Mathematics, University of
Mary Washington, Fredericksburg, VA 22401.
Email: {\tt kmelling@umw.edu}} \quad Dhruv Mubayi
\thanks{ Department of Mathematics, Statistics, and Computer
Science, University of Illinois,  Chicago, IL 60607. Research
supported in part by NSF grants DMS-0653946 and DMS-0969092. Email: {\tt
mubayi@math.uic.edu}} \quad Jacques Verstra\"ete \thanks{Department
of Mathematics, University of California, San Diego, La Jolla CA
92093-0112. Email: {\tt jacques@ucsd.edu} Research supported in part by
NSF grant DMS-0800704 and a Hellman Fellowship}\\ }

\date{\today}
\maketitle

\vspace{-0.3in}

\begin{abstract}

Fix integers $n \ge r \ge 2$. A clique partition of ${[n] \choose r}$
is a collection of proper subsets $A_1, A_2, \ldots, A_t \subset [n]$
such that $\bigcup_i{A_i \choose r}$ is a partition of ${[n] \choose
r}$.

 Let $\cp(n,r)$ denote the minimum size of a clique partition of
 ${[n] \choose r}$.
 A classical theorem of de Bruijn and Erd\H os states
that $\cp(n, 2) = n$.
In this paper we study $\cp(n,r)$, and show in
general that for each fixed $r \geq 3$,
\[ \cp(n,r) \geq (1 + o(1))n^{r/2} \quad \quad \mbox{ as }n
 \rightarrow \infty.\]
We conjecture $\cp(n,r) = (1 + o(1))n^{r/2}$.
This conjecture has already been verified (in a very strong sense) for $r = 3$
by Hartman-Mullin-Stinson. We give further evidence of this conjecture by constructing, for each $r
\ge 4$, a family of $(1+o(1))n^{r/2}$ subsets of $[n]$ with the following property:
no two $r$-sets of $[n]$ are covered more than once and all but $o(n^r)$ of the $r$-sets of $[n]$ are covered.

We  also give an absolute lower
bound $\cp(n,r) \geq {n \choose r}/{q + r - 1 \choose r}$
when $n = q^2 + q + r - 1$, and for each $r$ characterize the finitely many configurations achieving equality with the lower bound.
Finally we note the connection
of $\cp(n,r)$ to extremal graph theory, and determine some new asymptotically sharp bounds for the Zarankiewicz problem.
\end{abstract}

\section{Introduction}

A classical theorem of de Bruijn and Erd\H os~\cite{EdB} states that
the minimum number of proper complete subgraphs (henceforth cliques) of
the complete graph $K_n$ that are needed to partition its edge set is $n$.
Equality holds only for $n-1$ copies of $K_2$ on a vertex together with
a clique of order
$n-1$ (a near-pencil) or the cliques whose vertex sets are the $q^2 + q + 1$ point sets of lines in a projective plane of order
$q$ when $n=q^2+q+1$ and a projective plane exists.

\bigskip

In this paper, we  study the  analog of the de Bruijn-Erd\H os theorem for hypergraphs.  Write ${X \choose r}$ for the collection of $r$-elements subsets of a set $X$; this is a clique.
Throughout this paper, we write $[n] := \{1,2,\dots,n\}$. We will associate a hypergraph with its edge set.

\bigskip

{\bf Definition.}
Fix integers $n \ge r \ge 2$. A {\em clique partition} of a hypergraph $H$ is a collection of proper subcliques that partition $H$. Let
{\rm \cp}$(n, r)$ be the minimum size of a clique partition of ${[n] \choose r}$. In other words, {\rm \cp}$(n,r)$ is the minimum $t$ for which there are proper subsets $A_1, A_2, \ldots, A_t \subset [n]$ such that $\bigcup_i{A_i \choose r}$ is a partition of ${[n] \choose r}$.

\bigskip

The de Bruijn-Erd\H os Theorem can now be restated as $\cp(n,2)=n$ together with the characterization of equality. In this paper,
we consider $\cp(n,r)$ for $r>2$.
  A minimum clique partition of ${[n] \choose r}$ shall be referred to as an {\em optimal clique partition}.

As noted above, there are essentially two types of configurations achieving equality in the de Bruijn-Erd\H os theorem: near pencils and projective planes. For $r=3$
and $n=q^2+1$ there is only one
type of configuration that achieves equality, and this is called an inversive plane.

 An {\em inversive plane} is a pair $(V,{\cal C})$
where $V$ is a set of points and ${\cal C}$ is a set of subsets of points called circles satisfying the following axioms:

\medskip

\begin{center}
\begin{tabular}{lp{5.5in}}
(1) & any three points of $V$ must lie in exactly one circle in $C \in {\cal C}$, \\
(2) & four points in $V$ must exist that are not contained in a common circle and \\
(3) & for any circle $C \in {\cal C}$ and points $p \in C$ and $q \in V \backslash C$, there is exactly one circle $D \in {\cal C}$ such that $p,q \in
D$ and $C \cap D = \{p\}$.
\end{tabular}
\end{center}

\medskip

It is known that if an inversive plane on $n$ points exists, then $n$ is necessarily of the form $q^2 + 1$ for some integer $q \geq 2$, and the number $q$ is called
the {\em order} of the inversive plane.

If the cliques in a clique partition form an inversive plane, then we identify the partition with the inversive plane. The problem of determining $\cp(n,3)$ is quite well understood due to the following theorem.

\begin{theorem} \label{hms} {\bf (Hartman-Mullin-Stinson \cite{HMS})}
 Let $q \geq 3$ be an integer and $n = q^2 + 1$. Then
 $$\mbox{\rm \cp}(n,3)\ge n\sqrt{n-1}.$$ If $q$ is a prime power, then

  $\bullet$ {\rm \cp}$(n,3)=qn=n\sqrt{n-1}$, and

  $\bullet$ if $P$ is a clique partition of ${[n] \choose 3}$, $|P|=qn$, then $P$ is an inversive plane of order $q$.

  Consequently, as $n \rightarrow \infty$
$$\hbox{\rm \cp}(n,3) \sim n^{3/2}.$$
  \end{theorem}

In Section 2, we shall give a simple description of inversive planes requiring only basic field arithmetic.  There, we provide what might be considered a ``classical model'' for a finite inversive plane, starting with the coordinates for an affine plane over a finite field.  The model is not new (see, for instance the very complete book by Benz \cite{Benz2}, or the more recent survey by Wilker \cite{wilker}).
Although this treatment is not new, it appears that it is not well-known. We include it here for completeness.

In  \cite{HMS},
$\cp(n,3)$ is determined when $n$ is close to $q^2 + 1$ and
$q$ is a prime power, due to the existence of inversive planes and
classification~\cite{To} of linear spaces with $n$ points and $m$ blocks
such that $(m - n)^2 \leq n$. In general, however, it is likely to be
challenging to determine $\cp(n,3)$ exactly for all $n$, in contrast to
the de Bruijn-Erd\H os Theorem.

\subsection{Clique partitions of ${[n] \choose r}$}

The problem of determining $\cp(n,r)$ for $r > 3$ appears to be difficult. First we present a lower bound for $\cp(n,r)$ in terms of $\cp(n-1,r-1)$:

\begin{theorem}\label{lowerbound}
Let $r > 2$ and $n > r$ be integers. If {\rm \cp}$(n-1,r-1) = \delta${\rm \cp}$(n,r)$ then
\begin{equation}\label{bound}
\mbox{\rm \cp}(n,r) {\delta n \choose r} \leq {n \choose r}.
 \end{equation}
 Consequently, for each fixed $r \ge 2$, and $n \rightarrow \infty$,
$$ \hbox{ \rm\cp}(n,r) \geq (1 - o(1))n^{r/2}.$$
 \end{theorem}

Via constructions we shall give evidence that this lower bound
may be the true asymptotic behavior of $\cp(n,r)$ for every fixed
$r \geq 2$.

\begin{theorem} \label{t3}
For every fixed $r \geq 2$
there is a family $H \subseteq {[n] \choose r}$ of $r$-element sets
such that $|H| = {n \choose r} - o(n^r)$
and $H$ has a clique partition with $(1 + o(1))n^{r/2}$ cliques.
\end{theorem}

It appears to be difficult to extend the constructions for Theorem
\ref{t3} to full partitions of ${[n] \choose r}$ without adding many
more subcliques. We nevertheless conjecture that in general, the above
lower bound is asymptotically sharp:

\begin{conjecture} \label{conj}
For every fixed $r\ge2$ we have {\rm \cp}$(n, r) = (1 + o(1))n^{r/2}$ as $n \rightarrow \infty$.
\end{conjecture}

\setcounter{theorem}{3}

\subsection{Characterization of clique partitions}

Equality holds in the de Bruijn-Erd\H os Theorem for projective planes
and near-pencils. The nature of clique partition numbers $\cp(n,r)$
for $r > 2$ surely depends on number theoretic properties of $n$, so
unlike in the de Bruijn-Erd\H os Theorem, a characterization of optimal
partitions for all $n$ is likely to be more difficult.

Towards this goal, we prove the following theorem, which considers
values of $n$ that are parameterized in a way that includes the de
Bruijn-Erd\H os Theorem as a special case.  As is customary, we define
the binomial coefficient ${x \choose r}=x(x-1)\cdots (x-r+1)/r!$ for
any positive integer $r$ and real number $x$.  Also, recall that a {\em
Steiner $(n,k,t)$-system} is a collection of $k$-element sets of $[n]$
such that every $t$-element subset of $[n]$ lies in precisely one of
the $k$-element sets.
\begin{theorem}\label{lowerbound2}
Let $r \geq 2$ and let $n$ be a positive integer. Define the positive real number $q$ by the equation $n = q^2 + q + r - 1$. Then
\begin{equation}\label{bound2}
\mbox{\rm \cp}(n,r) \geq \frac{{n \choose r}}{{q + r - 1 \choose r}}.
\end{equation}
Equality holds if and only if one of the following holds:
\begin{center}
\begin{tabular}{lp{5in}}
$\bullet$ & $n = r + 1$ and the partition is ${[r + 1] \choose r}$ \\
$\bullet$ & $n > r + 1$ and $r = 2$ and the partition is a projective plane of order $q$ or a near pencil  \\
$\bullet$ & the partition is a Steiner $(n,k,r)$-system where $(n,k,r) \in \{(8,4,3),(22,6,3),(23,7,4),(24,8,5)\}$.
\end{tabular}
\end{center}
 \end{theorem}

Note that in the third case above
 such Steiner systems are known to exist for $r
\in \{3,4,5\}$, and so for those values of $r$ we have a complete
characterization of equality in Theorem \ref{lowerbound2}. The case
$r = 2$ is the de Bruijn-Erd\H os Theorem. We note that not a single
construction of a Steiner $(n,m,r)$-system is known for any $n>m > r >
5$. This is considered to be one of the major open problems in design
theory~\cite{vW}.

\subsection{Zarankiewicz problem}

There is a tight connection between \cp$(n,r)$ and the Zarankiewicz problem from extremal graph theory.
The {\em Zarankiewicz number} $z(m,n,s,t)$ denotes the maximum possible number of $1$s in an $m \times n$ matrix containing no $s \times t$ minor
consisting entirely of $1s$. This can be rephrased in terms of the maximum number of edges in an $m \times n$ bipartite graph containing no
complete bipartite subgraph with $s$ vertices in the part of size $m$ and $t$ vertices in the part of size $n$.
Our results for the  clique partition number $\cp(n,r)$ imply the following new asymptotically sharp results for Zarankiewicz numbers.

\begin{theorem} \label{z} Fix $r \geq 3$.
If $m = (1 + o(1)) n^{r/2}$, then
$z(m,n,2,r) = (1 + o(1))\sqrt{n}\,m$ as $n \rightarrow \infty$.
\end{theorem}

The special case $r=3$ had earlier been shown by Alon-R\'onyai-Szab\'o
\cite{ARS}, as part of a more general result motivated by
a question in discrepancy theory posed
by Matou\'sek.  However, our constructions here are different.

Theorem \ref{z} is in contrast to  a result of F\"uredi~\cite{F}
which shows that for each fixed $r \ge 2$, we have $z(m,n,2,r) =
(1 + o(1))\sqrt{(r-1)n}\,m$ whenever $m = (1 + o(1))n$.  The problem
of determining the asymptotic behavior of $\cp(n,r)$ for fixed $ r\ge
2$  seems more challenging than that of
determining the  Zarankiewicz numbers
$z(m,n,2,r)$ for $m = (1 + o(1))n^{r/2}$.
\bigskip

\subsection{Organization}

In Section 2, we discuss inversive planes and give a self-contained
presentation of inversive planes.
In Section 3, we discuss $\cp(n,r)$ for $r > 3$,
starting in Section 3.1 with the
proof of Theorem \ref{lowerbound},
proceeding in Section 3.2 with the proof of Theorem \ref{t3}, and
mentioning an alternative construction for $r \in \{4,5\}$ in
Section 3.3.
In Section 4, we prove Theorem \ref{lowerbound2}, and in Section 5 we
point out the connection to the Zarankiewicz problem and prove
Theorem \ref{z}.
\section{Inversive planes}

In this section, we give an elementary construction of inversive planes of prime power order $q \equiv 3$ (mod $4$) which may be of independent interest.  As mentioned in the introduction, this presentation is not new, but it appears  not to be well-known.

Recall that the blocks of an inversive plane $\pi$ are called circles and that if $\pi$ has $n$ points, then $n$ is necessarily of the form $q^2 + 1$ for some integer $q \geq 2$, and the number $q$ is called
the {\em order} of $\pi$.
In an inversive plane of order $q$, it is well-known that every circle has $q + 1$ points, every
point is in $q(q + 1)$ circles, and the total number of circles is $q^3 + q$ (see~\cite{DH} for more information on inversive planes).

The issue of constructing inversive planes of order $q$ has quite a long history beginning in the 1930s~\cite{W,B,Benz,D} and it is known that inversive planes of all prime power orders exist.
We will now present an elementary presentation of inversive planes.

{\bf Construction.}
For a prime power $q \equiv 3$ (mod 4) and $n = q^2 + 1$, we consider the vertex set $[n]$ as $(\field \times \field) \cup \{v\} := \field^2 \cup \{v\}$.
For $a=(a_1, a_2) \in \field^2$ and $\lambda \in \F_q\setminus \{0\}$, define the circle with center $a$ and finite radius $\lambda$ to be
$$C(a,\lambda)=\{(x_1, x_2) \in \field^2: (x_1-a_1)^2+(x_2-a_2)^2=\lambda\}.$$
There are $q^3 - q^2$ such circles, since there are $q^2$ choices for $a \in \field^2$ and then $q - 1$ choices for $\lambda$.

For each $a \in \field^2$ and $\mu \in \F_q$ define the following sets:
\begin{eqnarray*}
C(a) &=& \{(x_1,x_2)\in \field^2: x_2 - a_2 = a_1 x_1\} \cup \{v\} \\
C(\mu) &=& \{(x_1,x_2) \in \field^2 : x_1 = \mu\} \cup \{v\}.
\end{eqnarray*}
It is convenient to refer to these sets as circles too. Note that
each of them
has $q + 1$ points. These $q^2 + q$ special circles are in one-to-one
correspondence with the affine lines of $\field^2$; just add $\{v\}$,
a point at infinity, to each of the affine lines. The circles $C(a)$
come from lines with finite slope, whereas the circles $C(\mu)$ come
from lines with infinite slope. \qed

The total number of sets defined in our construction is $q^3 + q$. It remains to show that all the circles together form an inversive plane of order $q$, by
verifying the three axioms. In the Euclidean plane, every three points determine a unique circle unless they are collinear, and the basis for our construction is that this  remains true in finite fields. We make this precise in the next result.

\begin{lemma} Every three non-collinear points in $\field^2$ lie in a unique circle $C(a, \lambda)$.  No three collinear points lie on a circle $C(a, \lambda)$.
\end{lemma}

\proof Let $x=(x_1, x_2), y=(y_1, y_2), z=(z_1, z_2)$ be distinct
non-collinear points in $\field^2$.  We will show that there is a unique circle $C(a, \lambda)$ that contains all three of them with $\lambda \in \F_q\setminus \{0\}$.  We wish to determine the number of solutions $(a, \lambda)$ where $a \in \field^2$ and $\lambda \in \F_q\setminus\{0\}$ to the equations
\begin{align}(x_1-a_1)^2+(x_2-a_2)^2&=\lambda \label{1} \\
(y_1-a_1)^2+(y_2-a_2)^2 &=\lambda \label{2}\\
(z_1-a_1)^2+(z_2-a_2)^2 &=\lambda \label{3}
\end{align}
Now (\ref{1})$-$(\ref{2}) and (\ref{2})$-$(\ref{3}) give
$$ \left( \begin{array} {cc}
2(x_1-y_1) & 2(x_2-y_2) \\
2(y_1-z_1) & 2(y_2-z_2) \end{array} \right)
\left( \begin{array} {c} a_1 \\ a_2 \end{array} \right)
=\left( \begin{array} {c}  x_1^2-y_1^2+x_2^2-y_2^2 \\ y_1^2-z_1^2+y_2^2-z_2^2 \end{array} \right).$$
Since $x, y, z$ are non-collinear, the coefficient matrix above is invertible, and hence there is a unique $a=(a_1, a_2)$ that satisfies the matrix system above. For this particular choice of $a$, define
$\lambda$ using (\ref{1}).
It is straightforward to see that $\lambda$ automatically
satisfies both (\ref{2}) and (\ref{3}) and therefore
$x, y$ and $z$ determine a unique circle.

It remains to show that $\lambda\ne 0$.
If $\lambda=0$, then since $q \equiv 3$ (mod 4) we may apply Euler's
Theorem which says that -1 is not a square modulo $q$. Consequently,
the only solution to equations (\ref{1}), (\ref{2}) and (\ref{3}) is $x =
y=(a_1, a_2)$.  This contradicts the fact that $x,y,z$ are distinct and
therefore $\lambda \ne 0$.

Now we show that if $C(a, \lambda)$ is a circle then $|C(a,\lambda) \cap C(b)| \le 2$ for $b = (b_1,b_2) \in \field^2$ and also
$|C(a,\lambda) \cap C(\mu)| \leq 2$ for $\mu \in \F_q$. In the first instance, $(x_1, x_2) \in C(a,\lambda)$ means that
$$(x_1-a_1)^2+(x_2-a_2)^2=\lambda.$$
Substituting $x_2=b_1 x_1+b_2$ above gives the quadratic equation
$$(1+b_1^2)x_1^2 + c_1 x_1+c_2=0$$
for some $c_1, c_2 \in \F_q$.  Since $q \equiv 3$ (mod 4), Euler's Theorem implies that $1 + b_1^2 \ne 0$ and hence the quadratic above has at most two solutions.  Each of these solutions gives a unique solution for $x_2$ and consequently $|C(a,\lambda) \cap C(b)| \le 2$ as required. The proof that $|C(a,\lambda) \cap C(\mu)| \leq 2$
is similar.  \qed

The above lemma gives a family $P$ of $q^3 + q$ sets in ${[n]
\choose 3}$ when $n = q^2 + 1$ with the property that every set of
three distinct points in $[n]$ is covered by exactly
one set. Indeed, the Lemma clearly shows this for any three points
in $\field^2$, since if they are non-collinear they lie in a unique
circle $C(a, \lambda)$, and if they are collinear they lie in precisely
one affine line $C(a)\setminus\{v\}$ or $C(\mu)\setminus\{v\}$. If
the three points are of the form $\{x,y,v\}$ with $x,y \in \field^2$,
then since $\{x,y\}$ lies in a unique affine line, $\{x,y,v\}$ lies in
the unique extension of this line that has the form $C(a)$ or $C(\mu)$.

  We now show that $P$  is an inversive plane. By construction,  axioms
  1 and 2 are satisfied. To show that axiom 3 is satisfied we need an
  elementary counting argument using the fact that every circle has size
  $q+1$. This will be proved via the following lemma which shows that
  each circle $C(a, \lambda)$ has $q + 1$ points:

\begin{lemma}
Let $q \equiv 3$ mod 4 be a prime power, and let $(a,b) \in \field^2, \lambda \in \field \setminus\{0\}$. Then the number of solutions $(x,y) \in \field^2$ to the equation $(x-a)^2 + (y-b)^2 = \lambda$
is exactly $q + 1$.
\end{lemma}

\proof
Let $S_q = |\{(x,y) \in \field^2: (x-a)^2 + (y-b)^2 = \lambda\}|$.
By translation and scaling, $S_q = |\{(x,y)\in \field^2 : x^2 + y^2 = \delta\}|$ where
$\delta = \chi(\lambda) \in \{-1,1\}$ and $\chi$ is the quadratic character of $\F_q$. Therefore
\begin{eqnarray*}
S_q &=& \sum_{x + y = \delta} (1 + \chi(x))(1 + \chi(y)) \\
&=& q + \sum_{x \in F_q} \chi(x(\delta - x)))\\
&=& q + \sum_{x \in F_q \backslash \{\delta\}} \chi\Bigl(\frac{x}{\delta - x}\Bigr)\\
&=& q + \sum_{x \in F_q \backslash \{\delta\}} \chi\Bigl(-1 + \frac{\delta}{\delta - x}\Bigr)\\
&=& q + \sum_{w \in F_q \backslash \{-1\}} \chi(w) \; \; = \; \; q - \chi(-1).
\end{eqnarray*}
Here we used that $-1 + \delta/(\delta - x)$ is a permutation of $\F_q \backslash \{-1\}$ and
that $\sum_{w \in \F_q} \chi(w) = 0$. Since $q \equiv 3$ mod $4$, $\chi(-1) = -1$ and so $|S_q| = q + 1$
\qed

We now show that axiom 3 is satisfied. So let $C(a, \lambda)$ be a circle, $u \in C(a, \lambda)$ and $v \not\in C(a, \lambda)$. For each point $x \in C(a, \lambda)\setminus \{u\}$ we have a circle $C_x$ that contains $u,v,x$. Moreover, $C_x  \cap C_{x'}=\{u,v\}$ for $x \ne x'$.
Since these $q$ circles are disjoint outside $\{u,v\}$ and  they all have size $q+1$,  their union has size $2+(q-1)q=q^2-q+2=n-(q-1)$.
Therefore we can find a point $z$ outside the union of these circles.  But there is a circle $C$ that contains $u,v,z$ and so $C \cap C(a, \lambda)=\{u\}$. Moreover, $C$ contains all the remaining $q-1$ points above, so it is  the unique circle that contains $u,v$ and intersects $C(a,\lambda)$ at $\{u\}$.  This proves axiom 3.

\emph{}
\section{Clique partitions of ${[n] \choose r}$}

In this section, we prove Theorems \ref{lowerbound} and \ref{t3}.

\subsection{Proof of Theorem \ref{lowerbound}}

Let $P$ be an optimal clique partition of ${[n] \choose r}$ which has vertex set $[n]$. Then $P_v^* = \{C\setminus \{v\} : v \in C \in P\}$ is a clique partition of ${[n-1] \choose r - 1}$ so $|P_v^*| \geq \cp(n-1,r-1)$. Noting the identity
\begin{equation}
 \sum_{C \in P} |C| = \sum_{v \in [n]} |P_v^*|, \notag
 \end{equation}
we have
\begin{equation}\label{degrees}
 \sum_{C \in P} |C| \geq n\, \cp(n-1,r-1).
\end{equation}
On the other hand since $P$ is a clique partition of ${[n] \choose r}$,
\begin{equation}\label{cover}
 \sum_{C \in P} {|C| \choose r} = {n \choose r}.
 \end{equation}
By convexity of binomial coefficients, the right hand side is a minimum when the $|C|$ all equal their average value,
which by (\ref{degrees}) is at least $n\cp(n-1,r-1)/\cp(n,r) = \delta n$ by definition of $\delta$. Inserting this in (\ref{cover}) gives (\ref{bound}). \qed

\subsection{Proof of Theorem \ref{t3}}

By the known results about the distribution of primes it suffices
to prove the result for $n=q^2$, where $q$ is a prime power. For
other values of $n$ we can then take the construction for the
smallest $n'>n$ satisfying
$n'=q^2$, with $q$ being a prime power, and consider the induced
constructions on $n$ vertices among those $n'$.

Suppose, thus, that $n=q^2$, and let $F=GF(q)$ denote the finite
field of size $q$. Let $V$ be the set of all ordered pairs $(x,y)$
with $x,y \in F$. For every polynomial $p(x)$ of degree at most
$r-1$ over $F$, let $C_p$ denote the subset $C_p=\{(x,p(x)): x \in
F \}$ of $V$. The collection ${\cal C}$ of all sets $C_p$
is a collection of $q^r=n^{r/2}$ subsets of $V$, each of size
$q=\sqrt n$. Clearly, no two members of ${\cal C}$ share more than
$r-1$ elements, as two distinct polynomials of degree at most $r-1$
can share at most $r-1$ points. Moreover, every set $\{(x_i,y_i): 1
\leq i \leq r \}$ of $r$ points of $V$ in which the $x_i$-s are
pairwise distinct is contained in a unique set $C_p$, as there is a
unique polynomial $p$ of degree at most $r-1$ satisfying
$p(x_i)=y_i$ for all $i$. Let $H$ be the set of all $r$-subsets  of $V$
contained in a member of ${\cal C}$. Then
$H$ has a clique partition with the $n^{r/2}$ cliques $C_p$, and
the only $r$-sets in $V$ that do not belong to $H$ are those that
have at least two points $(x,y)$ with the same  first coordinate.
The number of these $r$-sets is at most
$$
(1+o(1))q {q \choose 2} {q-1 \choose {r-2}}q^{r-2}=
(1+o(1)) \frac{1}{2(r-2)!}  q^{2r-1}=O(n^{r-1/2})=o(n^r),
$$
completing the proof. \qed


\subsection{An alternative construction for  $r\in \{4,5\}$}
\label{subsec:ProofOf45}

In this subsection we present an alternative construction that
proves the assertion of Theorem \ref{t3} for $r \in \{4,5\}$.
Although this is less general than the previous construction we
believe it is interesting and may provide some extra insight.

We work in the classical finite projective plane of order $q$, denoted
by $\pi = PG(2,q)$. The plane $\pi$ contains $q^2+q+1$ points that can be
represented by homogeneous coordinates $(x,y,z)$.  A non-degenerate conic
of $\pi$ is a collection of points whose homogeneous coordinates satisfy
some non-degenerate quadratic form, and it is well-known that there is
exactly one such conic in $\pi$, up to isomorphism.  A typical example
is the set of points satisfying the form $y^2 = xz$ which contains the
points $\{(0,0,1)\} \cup \{(1,x,x^2):x \in \F_q\}$.  An arc is a set
of points, no three collinear, and it is straightforward to show that
conics form arcs.  Moreover, a classical result of Segre \cite{Segre}
says that when $q$ is odd, every set of $q+1$ points, no three collinear,
is in fact a conic.

For an overview of results on conics and arcs in general, the reader is
referred to Chapters 7 and 8 of \cite{JWPH}. It is well-known that there
are precisely $q^5 - q^2$ conics in $PG(2,q)$ and that five points in
general position (that is, no three of which are collinear) determine a
unique conic.  We can use these facts to cover 4-sets and 5-sets of $\pi$.

{\bf The case $r = 4$.} Distinguish a special point $P$ of $\pi$ and
consider the set of all conics passing through $P$.  Our points will be
the points of $\pi \setminus \{P\}$.  Hence, $n = q^2+q$.  Let $X$ be
the number of conics passing through an arbitrary point $Q$ of $\pi$.
By the transitive properties of $Aut(\pi)$, it follows that $X$ is
independent of $Q$.  We count pairs $(Q,\CC)$ with $Q$ a point of the
conic $\CC$, by first counting $Q$, and then counting $\CC$.  This give us
$$ (q^2+q+1) X = (q^5-q^2)(q+1)  = q^2(q-1)(q^2+q+1)(q+1).$$

If follows that there are $X = q^4-q^2$ conics through a point of $\pi$.
The number of 4-sets in the set $\pi \setminus \{P\}$ is ${q^2+q
\choose 4}=  (\frac{1}{24} + o(1)) q^8$.  The number of 4-sets of $\pi
\setminus \{P\}$ covered by the conics through $P$ is $(q^4-q^2) {q
\choose 4} = (\frac{1}{24} + o(1)) q^8$ (the difference is asymptotic to
$\frac{10}{24}q^7$ which is of a lower order as $q \rightarrow \infty$).
Hence, we have shown that for $n=q^2+q$, there is a collection of $(1 -
o(1))q^4=(1-o(1))n^2$ proper edge-disjoint subcliques of ${[n] \choose
4}$ that cover all but $o(q^8)=o(n^4)$ edges of ${[n] \choose 4}$.

{\bf The case $r = 5$.} We can repeat our construction, using all points
and conics of $\pi$, to obtain a similar result for $r=5$.
The construction here is in fact simpler than the one
 for $r = 4$. The total number
of 5-sets of points in $\pi$ is ${q^2+q+1 \choose 5} = (\frac{1}{120}
+ o(1)) q^{10}$.  The number of 5-sets covered by conics is $(q^5-q^2)
\times {q+1 \choose 5} = (\frac{1}{120} + o(1)) q^{10}$ (the difference
is asymptotically $\frac{1}{12}q^9$ which is of lower order).  Hence,
we have shown that for $n=q^2+q+1$, there is a collection of $(1 -
o(1))q^5=(1-o(1))n^{5/2}$ proper edge-disjoint subcliques of ${[n] \choose
5}$ that cover all but $o(q^{10})=o(n^5)$ edges of ${[n] \choose 5}$.
This completes the (second) proof of
Theorem \ref{t3} for $r \in \{4,5\}.$ \qed

\section{Proof of Theorem \ref{lowerbound2}}

We first address the lower bound in Theorem \ref{lowerbound2}.
Before beginning the proof we need the following simple lemma,
known as Chebyshev Sum Inequality (c.f., e.g., \cite{HLP}.)

\begin{lemma}[Chebyshev Sum Inequality]
\label{aibi}
Suppose that $p>1$ is an integer and
$f, g: [p] \rightarrow {\mathbb Z^+}$ are non-decreasing functions. Then
$$\sum_{i=1}^p f(i)g(i) \ge \frac{1}{p}
 \sum_{i=1}^p f(i) \sum_{i=1}^p g(i).$$
\end{lemma}
There are several simple proofs of this inequality. A short one is
to observe that the right hand side is the expectation of
the sum $\sum_{i=1}^p f(i) g(\sigma(i))$, where $\sigma$ is a
random uniformly chosen permutation of $\{1,2,\ldots ,p\}$, and
that the maximum value of this sum is obtained when $\sigma$ is the
identity, since if $\sigma(i) > \sigma(j)$ for some $i<j$, then the
permutation $\sigma'$ obtained from $\sigma$ by swapping the values
of $\sigma(i)$ and $\sigma(j)$ satisfies
$\sum_{i=1}^p f(i) g(\sigma'(i)) \geq
\sum_{i=1}^p f(i) g(\sigma(i))$.

\subsection{Proof of the lower bound}

In this section we show the lower bound in Theorem \ref{lowerbound2}.
Define $(n)_r := n(n - 1) \dots (n - r + 1)$ and $\phi(n,r) = {n \choose r}/{q + r - 1 \choose r}$. Let $P$ be a clique partition of ${[n] \choose r}$ where $n = q^2 + q + r - 1$ and $q$ is a positive real number. We aim to show
$|P| \geq \phi(n,r)$, and we proceed by induction on $r$. For $r = 2$, this follows from the de Bruijn-Erd\H os Theorem since $\phi(n,2) = q^2 + q + 1 = n$.
Since $\phi(r+1,r) = r + 1$, we may assume $n > r + 1$. Suppose $r > 2$ and let $C_1,C_2,\dots,C_p$ be the cliques in $P$, of respective sizes
$n_1 \leq n_2 \leq \dots \leq n_p$. Note as in (\ref{degrees}),
\begin{equation}\label{degrees2}
 \sum_{i=1}^p n_i \geq n\phi(n-1,r-1).
\end{equation}
For a set $S \subset [n]$, let $P_S^* = \{C \setminus S : C \in P\}$. Then
\begin{equation}
 \sum_{i=1}^p {n_i\choose r-2} = \sum_{S \in {[n]\choose r-2}} |P_S^*|. \notag
 \end{equation}
 By the de Bruijn-Erd\H os Theorem, since $P_S^*$ is a clique partition of ${[n - r + 2] \choose 2}$, $|P_S^*| \geq n - r + 2$ and therefore
\begin{equation}\label{degrees3}
 \sum_{i=1}^p {n_i \choose r - 2} \geq {n \choose r - 2} \, (n - r + 2)
\end{equation}
with equality if and only if there is a projective plane on $n - r + 2$ points, $q$ is an integer, and every clique in $P$ has size $q + r - 1$. Note
here that we are ruling out the possibility of a near pencil in every $P_S^*$, otherwise $P$ would have a clique of size $n - 1$ and then $|P| \geq 1 + {n - 1 \choose r - 1} > \phi(n,r)$ since $n > r + 1$. On the other hand since $P$ is a clique partition of ${[n] \choose r}$,
\begin{equation}\label{cover2}
 \sum_{i=1}^p {n_i \choose r} = {n \choose r}.
 \end{equation}
Since $n_i \geq r$ for all $i$, we may apply Lemma \ref{aibi} with $f(i)=(n_i)_{r-2}$ and $g(i)=(n_i-r+2)(n_i-r+1)$ to obtain
\[ \sum_{i = 1}^p (n_i)_r \geq \frac{1}{p} \sum_{i=1}^p (n_i)_{r-2} \cdot \sum_{i=1}^p (n_i -r + 2)(n_i - r + 1).\]
By (\ref{degrees3}) and (\ref{cover2}), we have
\[ p \geq \frac{1}{n - r + 1} \sum_{i=1}^p (n_i - r + 2)(n_i - r + 1).\]
The function $h(x)=(x-r+2)(x-r+1)$ is convex for $x \ge r$. Hence Jensen's inequality yields
$$\sum_{i=1}^p (n_i - r + 2)(n_i - r + 1)\ge p(c-r+2)(c-r+1)$$
where $c$ is the average size of a clique in $P$. Consequently,
\[ n - r + 1 \geq (c - r + 2)(c - r + 1)\]
and since $n = q^2 + q + r - 1$, we find $c \leq q + r - 1$. Together with (\ref{degrees2}) we get
\[ p(q + r - 1) \geq n\phi(n-1,r-1).\]
Using  the identity
$\phi(n,r)(q + r - 1) = n\phi(n-1,r-1)$, this gives $p \geq \phi(n,r)$ and completes the proof of the lower bound in Theorem \ref{lowerbound2}.

\subsection{Equality in Theorem \ref{lowerbound2}}

For the constructions giving equality in Theorem \ref{lowerbound2}, we note that for $r = 2$ this is exactly the de Bruijn-Erd\H os Theorem.
Moving on to $r \geq 3$, we have $n = q^2 + q + r - 1$, and equality holds in the above arguments if and only if equality holds
in (\ref{degrees3}), which means that every $(r-2)$-set $S$ has
the property that $P^*_S$ is an optimal
 partition ${[n]\setminus S \choose 2}$. Therefore every
 $r$-set is covered at most (in fact exactly) once and $P$ is a Steiner
 $(n,q+r-1,r)$-system. It is well known (see~\cite{vW}) that the following
 divisibility requirements are
necessary for the existence of such Steiner systems:
\begin{equation}\label{divisibility}
{q + r - 1 - i \choose r-i}  \;\; \Big| \;\; {n-i\choose r-i} \quad \mbox{ for } i = 0,1,\dots,r-1.
\end{equation}
Equivalently, since $n = q^2 + q + r - 1$, (\ref{divisibility}) is equivalent to:
\begin{equation}\label{divisibility2}
\prod_{j=1}^{r - i} (q + r - i - j) \; \; \; \Big| \; \; \; \prod_{j=1}^{r - i} (q^2 + q + r - i - j) \quad \mbox{ for } i = 0,1,\dots,r-1.
\end{equation}
For each $r$, let $Q_r$ be the set of values of $q$ permitted by (\ref{divisibility2}). Then the $Q_r$ form a decreasing chain $Q_3 \supseteq Q_4 \supseteq \dots$
and $1 \in Q_r$ for all $r \geq 3$. Now $Q_3$ is the set of $q$ such that
\begin{eqnarray*}
q &|& (q^2 + q) \\
q(q + 1) &|& (q^2 + q + 1)(q^2 + q) \\
q(q + 1)(q + 2) &|& (q^2 + q + 2)(q^2 + q + 1)(q^2 + q).
\end{eqnarray*}
The first two conditions are trivially satisfied, and the last is
\[ (q + 2) \; | \; (q^2 + q + 2)(q^2 + q + 1).\]
This is equivalent to $(q + 2)|12$. It follows that $q \in \{1,2,4,10\}$, $Q_3 = \{1,2,4,10\}$ and then $n \in \{4,8,22,112\}$.
Since Steiner $(n,q+2,3)$-systems exist for $(n,q) \in \{(4,1),(8,2),(22,4)\}$ (see van Lint and Wilson~\cite{vW}),
Theorem \ref{lowerbound} is tight in those cases and
\[ \cp(4,3) = 4 \quad \cp(8,3) = 14 \quad \cp(22,3) = 77.\]
The case $q = 10$ may be ruled out, since it is accepted (see~\cite{Lam}) that a projective plane of order ten does not exist, hence it is impossible that $P^*_S$ is a partition of ${[n]\setminus S\choose 2}$ for $|S|=1$.
For $r = 4$, we have the same divisibility requirements as $r = 3$ together with
\[ q(q + 1)(q + 2)(q + 3) \; \; | \; \; (q^2 + q + 3)(q^2 + q + 2)(q^2 + q + 1)(q^2 + q) .\]
Equivalently, we are looking for $q \in Q_3$ satisfying
$$(q + 2)(q + 3) \; \; | \; \; (q^2 + q + 3)(q^2 + q + 2)(q^2 + q + 1).$$
This is not satisfied for $q \in \{2,10\}$ and hence the only values $q \in Q_3=\{1,2,4,10\}$ which satisfy these requirements are $q \in \{1,4\}$ and $Q_4 = \{1,4\}$.
Since a Steiner $(23,7,4)$-system exists which corresponds to $q = 4$, we have
\[ \cp(5,4) = 5 \quad \cp(23,4) = 253.\]
For $r = 5$, we are seeking Steiner $(q^2 + q + 4,q + 4,5)$-systems, which implies $q \in \{1,4\} = Q_5$ and $n \in \{6,24\}$. Since
a Steiner $(24,8,5)$-system exists, we have the complete solution
\[ \cp(6,5) = 6 \quad \cp(24,5) = 759.\]
Finally we show $Q_r = \{1\}$ for $r \geq 6$ to complete the proof. First we have $Q_5 = \{1,4\}$ as just seen, and since $Q_r \subseteq Q_5$ for all $r \geq 5$,
we only have to show $Q_6 = \{1\}$. If $4 \in Q_6$, then the divisibility requirement (\ref{divisibility2}) with $i = 0$ and $q = 4$ is
\[ 9 \cdot 8 \cdot 7 \cdot 6 \cdot 5 \cdot 4 \; | \; 25 \cdot 24 \cdot 23 \cdot 22 \cdot 21 \cdot 20.\]
This is false, since $3$ is a prime factor with multiplicity
three on the left, and only two on the right.
Therefore $Q_r = \{1\}$ for $r \geq 6$, and the only possible clique partition achieving equality in Theorem \ref{lowerbound2} has
$q = 1$ and $n = r + 1$, and therefore it must be ${[r + 1] \choose r}$. \qed

\section{Zarankiewicz Problem and Theorem \ref{z}}

In this section we point out the connection between Zarankiewicz
numbers and clique partitions.  As a byproduct, we prove Theorem \ref{z}.
Recall that the Zarankiewicz number $z(m,n,s,t)$ is the maximum number of edges in an $m \times n$ bipartite graph containing no
complete bipartite subgraph with $s$ vertices in the part of size $m$ and $t$ vertices in the part of size $n$.
The clique partition number $\cp(n,r)$ is related to the Zarankiewicz numbers in the following sense:

\begin{lemma} \label{zarankiewicz}
Let $n \geq r$ be a positive integer and $m = {n \choose r}/{k \choose r}$. Then there is a partition of ${[n] \choose r}$ into $m$ cliques of size $k$ if and only if $z(m,n,2,r) = km$.
\end{lemma}

\begin{proof}
We follow the classical argument of \cite{KST}.
For any $m \times n$ bipartite graph with parts $A$ and $B$,
not containing a $2 \times r$ complete bipartite subgraph
(with the two vertex set lying in $A$),
\begin{equation} \label{z1} \sum_{a \in A} {d(a) \choose r} \leq {|B| \choose r}.\end{equation}
By convexity, if $e$ is the number of edges in the graph we obtain
\begin{equation} \label{z2} m{e/m \choose r} \leq {n \choose r}.\end{equation}
Suppose $e = z(m,n,2,r)$ and $e > km$. Then the above formula becomes
\[ m{k \choose r} < {n \choose r}.\]
However this contradicts the identity relating $m,n$ and $k$ in the lemma. It follows that $z(m,n,2,r) \leq km$.

Now consider the incidence graph of points and sets of size $k$ in the clique partition of ${[n] \choose r}$.
This is an $m \times n$ bipartite graph, and since we have a clique partition, it does not contain a $2 \times r$ complete bipartite
subgraph. Therefore $z(m,n,2,r) \geq km$.

If $z(m,n,2,r) = km$,  then equality holds in (\ref{z1}) and (\ref{z2}), which means that $d(a)=k$ for every $a \in A$ and every $r$-set of $B$ is covered exactly once. This gives us a partition of ${[n] \choose r}$ into $m$ cliques of size $k$.
\qed
\end{proof}

The proof of Lemma \ref{zarankiewicz} can be easily rewritten to prove
the following: Fix $r \ge 2$  and let $H \subset {[n] \choose r}$
satisfy $|H| \ge {n \choose r} -o(n^r)$. Suppose that there is a partition
of $H$ into $m=(1+o(1)){n \choose r}/{k \choose r}$ cliques of size
$k$. Then $z(m, n, 2,r) = (1 + o(1)) km$.

Using the preceding results and constructions in Sections 2 and 4, we obtain the following theorem on Zarankiewicz numbers which proves Theorem \ref{z}.

\begin{theorem}
Let $n = q^2 + 1$ and $m = qn$ where $q$ is a prime power. Then
$z(m,n,2,3) = (q^2 + q) n$. Furthermore, if $r \geq 4$, $n$ is an
integer and $m = (1 + o(1))n^{r/2}$, then
$z(m,n,2,r) = (1 + o(1))\sqrt{n}\, m$ as $n \rightarrow \infty$.
\end{theorem}
\proof
In the first case we apply Lemma \ref{zarankiewicz} with $k=q+1$ and
$r=3$. Such a clique partition exists by the construction of inversive
planes.  Then we see that $m=qn={n \choose 3}/{q+1 \choose 3}$ and so
we immediately get $z(m,n,2,3)=km=(q+1)qn=(q^2+q)n$.

For the cases $r  \geq 4$ we use the asymptotic version of Lemma
\ref{zarankiewicz} above and the results in Section 3.2.
If $n=q^2$ with $q$ a prime power, the desired result follows from
the construction in the proof of Theorem \ref{t3}. For other values
of $n$ we take the smallest prime power $q$ so that $q^2 \geq n$,
consider the construction for that which is a bipartite graph with
vertex classes $A$ and $B$, and take the expected number of
edges in an induced subgraph on randomly chosen sets of sizes
$n$ and $m$ in $A$ and $B$, respectively.
\qed

\section{Concluding remarks}

$\bullet$ {\bf  Clique partitions and extensions of Fisher's
Inequality}

The de Bruijn-Erd\H{o}s Theorem is an extension of a
well known inequality of Fisher \cite{Fi} that asserts that any
nontrivial
clique partition of ${[n] \choose 2}$ in which  all cliques have the
same cardinality contains at least $n$ cliques. Fisher's Inequality
has been
extended in several ways. One such extension is due to
Ray-Chaudhuri and Wilson \cite{RW} who proved that for even $r$, any
clique partition of ${[n] \choose r}$ in which  all cliques have
the same cardinality contains at least ${n \choose {r/2}}$ cliques.
Our bound here is stronger (by a factor of $(1+o(1))(r/2)!$ for
even $r$), and applies to odd values of $r$ as
well without having to assume that all cliques have the same
cardinality. On the other hand, the result in \cite{RW}, whose
proof is algebraic, holds even
if every $r$-set is covered exactly $\lambda$ times for some
$\lambda>0$.

$\bullet$ {\bf  Partitions into complete $r$-partite $r$-graphs}

A well known result of Graham and Pollak \cite{GP} asserts
that the complete graph
on $n$ vertices can be edge partitioned into $n-1$, but not less,
complete bipartite graphs.  This can be viewed as a bipartite
analogue of the de Bruijn-Erd\H{o}s Theorem. In the bipartite case,
unlike the one dealing with clique-partitions, there are
many  extremal configurations, and the only known proofs of the
lower bound are algebraic (though recently Vishwanathan \cite{V} presented a counting argument that replaces the linear algebraic part of one of these proofs).

The $r$-partite
version of the problem considered here is to determine the minimum
possible number of complete $r$-partite $r$-graphs in a
decomposition of the edges of the complete $r$-graph on $n$
vertices. More formally, this is the minimum $p=p(n,r)$ so that there are
$p$ collections of the form $\{A^{(i)}_1, A^{(i)}_2, \ldots ,
A^{(i)}_r\}$,
$(1 \leq i \leq p)$, satisfying
$A^{(i)}_j \cap A^{(i)}_s =\emptyset$ for all $1 \leq i \leq p$ and all
$1 \leq j < s \leq r$, $A^{(i)}_j \subset [n]$ for all admissible
$i$ and $j$, and for every $r$-subset $R$ of $n$ there is a unique
$i$, $1 \leq i \leq p$ so that $|R \cap A^{(i)}_j|=1$ for all
$1 \leq j \leq r$. Thus, the Graham Pollak result asserts that
$p(n,2)=n-1$. In \cite{Al} it is proved that $p(n,3)=n-2$ and that
for every fixed $r$ there are two positive constants $c_1(r),c_2(r)$
so that $c_1(r)n^{\lfloor r/2 \rfloor}\le
p(n,r) \leq
c_2(r)n^{\lfloor r/2 \rfloor}$ for all $n$ (see also \cite{CKV} for slight improvements).  Note that for even $r$
the exponent of $n$  is the same as the one appearing in our bounds
for the clique partition problem (that is, the function $\cp(n,r)$
discussed here), but for odd values of $r$, and in particular for
$r=3$ where both functions are well understood, the exponents
differ.

$\bullet$ {\bf Geometric description of inversive planes.}

Inversive planes are more widely modeled using the techniques of finite projective geometry.  For completeness, we briefly describe the known models here since these planes provide optimal configurations.
Inversive planes can be constructed from ovoids. These are sometimes
called egg-like inversive planes (or Miquelian planes when the ovoid is an
elliptic quadric as we describe below). Thas~\cite{T} showed that for odd
$q$, if $P$ is an optimal clique partition of ${[n] \choose 3}$ and $n =
q^2 + 1$, and for some $v \in [n]$ the vertex sets $V(C) \backslash \{v\}
: C \in P$ constitute the point sets of the lines of the affine plane
$AG(2,q)$, then the vertex sets of the cliques in $P$ are the circles of
an egg-like inversive plane. It follows that the construction given in
Section 2.1 is equivalent to the known construction of inversive planes
when the field has odd characteristic $q$.

To describe the ovoidal construction, the setting is finite projective
3-space $PG(3,q)$.
An elliptic quadric of $PG(3,q)$ forms a set of $q^2+1$ points, no three of which
are collinear; in general, such a set of points in $PG(3,q)$ is called
an \emph{ovoid}.  Now consider the points
of an ovoid $\O$ of $PG(3,q)$ as the vertices of ${[q^2 + 1] \choose 3}$
together with the non-tangent planar cross sections of $\O$.  As three
points uniquely determine a plane, each set of three points of $\O$
(each edge of the hypergraph) is covered by one of these planar cross
sections. Moreover, it is straightforward to show that each non-tangent
plane meets $\O$ in $q+1$ points, forming a planar conic, an object we
described in Section \ref{subsec:ProofOf45}.  It follows that the
set of points of $\O$ together with the planar cross-sections of $\O$
form an inversive plane.

\bigskip

It is an important open question in projective geometry as to whether for any $q$, an inversive plane of order $q$ must be constructed from an ovoid as above. Barlotti~\cite{B} and
Dembowski~\cite{D} showed that if $q$ is odd, then any ovoid is projectively equivalent to an elliptic quadric -- the set of points satisfying a quadratic form $xy + \phi(w,z)$ (here $\phi$ is an irreducible quadratic form over $\F_q$). In this case, there is an alternative description: consider the ground set to be $\F_{q^2} \cup \{\infty\}$ and consider the images of the set $\F_q \cup \{\infty\}$ under the group of permutations
\[ PGL(2,q^2) = \Bigl\{z \mapsto \frac{a z^\alpha + b}{c z^\alpha + d} : a d - b c \neq 0, \alpha \in Aut(\F_{q^2})\Bigr\}.\]
The $q^3 + q$ images are taken to be the circles, each with $q + 1$ points, and it can be shown that the elliptic quadric construction is equivalent to the semi-linear fractional group construction.  Note that this model can also be viewed as the points of $PG(1,q^2)$ with blocks corresponding to the Baer sublines (i.e., copies of $PG(1,q)$) therein.

For $q$ even there exist ovoids that are not elliptic quadrics, the so-called Tits ovoids \cite{Tits}. However, Dembowski's Theorem~\cite{D} shows that if $q$ is even, then any inversive plane is derived from an ovoid as above, and so in this sense, optimal clique partitions of ${[n] \choose 3}$ when $n = q^2 + 1$ and $q = 2^h$ for $h \geq 2$ are unique up to the choice of the ovoid. It follows from the previously mentioned result of Thas~\cite{T}, translated into the language of clique partitions, that for odd $q$, if $P$ is an optimal clique partition of ${[n] \choose 3}$ and
$n = q^2 + 1$, and some $P_v^*$ is the classical affine plane $AG(2,q)$, then $P$ itself is constructed as above from an elliptic quadric. In particular,
the optimal clique partitions are unique. It appears difficult, however, to claim that any of the clique partitions $P_v^*$ are isomorphic
to the classical affine plane. Much more on finite projective 3-space, including results on ovoids and quadrics, can be found in the book by Hirschfeld \cite{JWPH2}.
\bigskip

$\bullet$ {\bf Classification of linear spaces.} Theorem \ref{lowerbound} gives a general lower bound on $\cp(n,r)$ in terms of $\cp(n-1,r-1)$. We concentrate on the case $r = 3$.
A necessary condition for tightness in Theorem \ref{lowerbound2} for
$r = 3$ when $n = q^2 + q + 2$ and $q$ is a positive integer is the
existence of a projective plane of order $q$ on $n - 1$ points. If $n$
is an integer such that no projective plane on $n - 1$ points exists, this
raises the question as to the minimum 2-designs which are not projective
planes and not near pencils on $n - 1$ points. A complete analysis of
2-designs on $v$ points with $b$ blocks such that $(b - v)^2 \leq v$
was carried out by Totten~\cite{T} (1976) (see also~\cite{Fo}). In
particular, those 2-designs are one of the following:
\begin{center}
\begin{tabular}{lp{5in}}
(a) &  near pencils \\
(b) & an affine plane of order $q$ with a linear space on at most $q + 1$ new points at infinity (add a point to every line in a parallel class, and then amongst the new points create a 2-design of lines -- this in particular contains the projective planes when $q + 1$ points are added and a single line through them is added) \\
(c) & embeddable in a projective plane (delete at most $v + 1$ points from a projective plane, deleting any line which becomes a singleton) \\
(d) & an exceptional configuration with $v = 6$ and $b = 8$.
\end{tabular}
\end{center}

We point out that Theorem \ref{lowerbound} can be tight. Suppose $n = 21$. By the classification of 2-designs, $|P_v^*|$ is at least the number of lines in an affine plane of order $q = 4$ together with the trivial linear space on $q$ points at infinity, so $\cp(n-1,2) \geq q^2 + q + 1$ i.e. $\cp(20,2) \geq 21$. Using this in Theorem \ref{lowerbound}, if $x = \cp(21,3)$ we obtain
\[ x \cdot {(21)^2/x \choose 3} \leq {21 \choose 3}\]
 which gives
\[ x \geq \frac{441}{676}(\sqrt{32361} - 63) > 76.\]
 Since $x$ is an integer, $x \geq 77$. However (see Theorem \ref{lowerbound2}) $\cp(22,3) = 77$ due to the existence of the Steiner $(22,6,3)$-system, so we conclude $x \leq \cp(22,3) = 77$.
 It follows that $\cp(21,3) = 77$. While it is possible to generalize the argument we just used for $n = 21$ by using the classification of 2-designs, it is the lack of constructive
upper bounds where more work is needed, and in general there is a gap between the upper and lower bounds for $\cp(n,3)$. An extensive survey of the existence problem
for Steiner systems may be found in~\cite{C}.

\end{document}